\documentclass[12pt]{amsart}
\usepackage{amsmath,amssymb,amsbsy,amsfonts,latexsym,amsopn,amstext,
                                               amsxtra,euscript,amscd,bm}
\usepackage{url}
\usepackage{cite}
\usepackage[colorlinks,linkcolor=blue,anchorcolor=blue,citecolor=blue,backref=page]{hyperref}
\usepackage{color}
\usepackage{graphics,epsfig}
\usepackage{graphicx}
\usepackage{float}

\hypersetup{breaklinks=true}

\renewcommand*{\backref}[1]{}
\renewcommand*{\backrefalt}[4]{%
    \ifcase #1 (Not cited.)%
    \or        (p.\,#2)%
    \else      (pp.\,#2)%
    \fi}

\begin{document}

\newtheorem{theorem}{Theorem}
\newtheorem{lemma}[theorem]{Lemma}
\newtheorem{claim}[theorem]{Claim}
\newtheorem{cor}[theorem]{Corollary}
\newtheorem{prop}[theorem]{Proposition}
\newtheorem{definition}{Definition}
\newtheorem{quest}[theorem]{Question}
\newtheorem{remark}[theorem]{Remark}
\newtheorem{conj}[theorem]{Conjecture}
\newcommand{\hh}{{{\mathrm h}}}

\numberwithin{equation}{section}
\numberwithin{theorem}{section}
% \numberwithin{remark}{section}
\numberwithin{table}{section}

\def\sssum{\mathop{\sum\!\sum\!\sum}}
\def\ssum{\mathop{\sum\ldots \sum}}
\def\iint{\mathop{\int\ldots \int}}

\def\squareforqed{\hbox{\rlap{$\sqcap$}$\sqcup$}}
\def\qed{\ifmmode\squareforqed\else{\unskip\nobreak\hfil
\penalty50\hskip1em\null\nobreak\hfil\squareforqed
\parfillskip=0pt\finalhyphendemerits=0\endgraf}\fi}%%

%  use the AMS-Euler Fraktur fonts
%%%%%%%%%%%%%%%%%%%%%%%%%%%%%%%%%%
\newfont{\teneufm}{eufm10}
\newfont{\seveneufm}{eufm7}
\newfont{\fiveeufm}{eufm5}
%%%%%%%%%%%%%%%%%%%%%%%%%%%%%%%%%
%
%  allow automatic size selection in math mode
%
%%%%%%%%%%%%%%%%%%%%%%%%%%%%%%%%%
\newfam\eufmfam
     \textfont\eufmfam=\teneufm
\scriptfont\eufmfam=\seveneufm
     \scriptscriptfont\eufmfam=\fiveeufm
%%%%%%%%%%%%%%%%%%%%%%%%%%%%%%%%%
%
%  \frak works on a single symbol at a time...
%
\def\frak#1{{\fam\eufmfam\relax#1}}

\def\t{\widetilde}

\def\fF{\EuScript{F}}
\def\tB{\widetilde{B}}

\newcommand{\bflambda}{{\boldsymbol{\lambda}}}
\newcommand{\bfmu}{{\boldsymbol{\mu}}}
\newcommand{\bfxi}{{\boldsymbol{\xi}}}
\newcommand{\bfrho}{{\boldsymbol{\rho}}}

\def\fK{Frak K}
\def\fT{Frak{T}}

\def\fA{{Frak A}}
\def\fB{{Frak B}}
\def\fC{\mathfrak{C}}

\def \balpha{\bm{\alpha}}
\def \bbeta{\bm{\beta}}
\def \bgamma{\bm{\gamma}}
\def \blambda{\bm{\lambda}}
\def \bchi{\bm{\chi}}
\def \bphi{\bm{\varphi}}
\def \bpsi{\bm{\psi}}

\def\eqref#1{(\ref{#1})}

\def\vec#1{\mathbf{#1}}

%\def\squareforqed{\hbox{\rlap{$\sqcap$}$\sqcup$}}
%\def\qed{\ifmmode\squareforqed\else{\unskip\nobreak\hfil
%\penalty50\hskip1em\null\nobreak\hfil\squareforqed
%\parfillskip=0pt\finalhyphendemerits=0\endgraf}\fi}

%%%%%%%%%%%%%%%%%%%%%%%%%
% Alphabet calligraphie %
%%%%%%%%%%%%%%%%%%%%%%%%%
\def\cA{{\mathcal A}}
\def\cB{{\mathcal B}}
\def\cC{{\mathcal C}}
\def\cD{{\mathcal D}}
\def\cE{{\mathcal E}}
\def\cF{{\mathcal F}}
\def\cG{{\mathcal G}}
\def\cH{{\mathcal H}}
\def\cI{{\mathcal I}}
\def\cJ{{\mathcal J}}
\def\cK{{\mathcal K}}
\def\cL{{\mathcal L}}
\def\cM{{\mathcal M}}
\def\cN{{\mathcal N}}
\def\cO{{\mathcal O}}
\def\cP{{\mathcal P}}
\def\cQ{{\mathcal Q}}
\def\cR{{\mathcal R}}
\def\cS{{\mathcal S}}
\def\cT{{\mathcal T}}
\def\cU{{\mathcal U}}
\def\cV{{\mathcal V}}
\def\cW{{\mathcal W}}
\def\cX{{\mathcal X}}
\def\cY{{\mathcal Y}}
\def\cZ{{\mathcal Z}}
\newcommand{\rmod}[1]{\: \mbox{mod} \: #1}

\def\cg{{\mathcal g}}

\def\vr{\mathbf r}

\def\e{{\mathbf{\,e}}}
\def\ep{{\mathbf{\,e}}_p}
\def\em{{\mathbf{\,e}}_m}

\def\Tr{{\mathrm{Tr}}}
\def\Nm{{\mathrm{Nm}}}

 \def\SS{{\mathbf{S}}}

\def\lcm{{\mathrm{lcm}}}
\def\ord{{\mathrm{ord}}}

\def\({\left(}
\def\){\right)}
\def\fl#1{\left\lfloor#1\right\rfloor}
\def\rf#1{\left\lceil#1\right\rceil}

\def\mand{\qquad \mbox{and} \qquad}

\newcommand{\commF}[1]{\marginpar{%
\begin{color}{blue}
\vskip-\baselineskip %raise the marginpar a bit
\raggedright\footnotesize
\itshape\hrule \smallskip F: #1\par\smallskip\hrule\end{color}}}

\newcommand{\commI}[1]{\marginpar{%
\begin{color}{magenta}
\vskip-\baselineskip %raise the marginpar a bit
\raggedright\footnotesize
\itshape\hrule \smallskip I: #1\par\smallskip\hrule\end{color}}}

%%%%%%%%%%%%%%%%%%%%%%%%%%%%%%%%%%%%%%%%%%%%%%%%%%%%%%%%
%%%%%%%%%%%%%%%%%%%%%%%%%%%%%%%%%%%%%%%%%%%%%%%%%%%%%%%%
%%%%%%%%%%%%%%%%%%%%%%%%%%%%%%%%%%%%%%%%%%%%%%%%%%%%%%%%
%%%%%%%%%%%%%%%%%%%%%%%%%%%%%%%%%%%%%%%%%%%%%%%%%%%%%%%%

%%%%%%%  END OF STANDARD STUFF %%%%%%%%%

%%%%%%%%%%%%%%%%%%%%%%%%%%%%%%%%%%%%%%%%%%%%%%%%%%%%%%%%
%%%%%%%%%%%%%%%%%%%%%%%%%%%%%%%%%%%%%%%%%%%%%%%%%%%%%%%%
%%%%%%%%%%%%%%%%%%%%%%%%%%%%%%%%%%%%%%%%%%%%%%%%%%%%%%%%
%%%%%%%%%%%%%%%%%%%%%%%%%%%%%%%%%%%%%%%%%%%%%%%%%%%%%%%
%%%%%%%%%%%
%%% Spell

\hyphenation{re-pub-lished}

\mathsurround=1pt

\def\bfdefault{b}
\overfullrule=5pt

\def \D{{\mathbb D}}
\def \T{{\mathbb T}}
\def \F{{\mathbb F}}
\def \K{{\mathbb K}}
\def \N{{\mathbb N}}
\def \Z{{\mathbb Z}}
\def \Q{{\mathbb Q}}
\def \R{{\mathbb R}}
\def \C{{\mathbb C}}
\def\Fp{\F_p}
\def \fp{\Fp^*}

\def\Kmn{\cK_p(m,n)}
\def\psmn{\psi_p(m,n)}
\def\SI{\cS_p(\cI)}
\def\SIJ{\cS_p(\cI,\cJ)}
\def\SAIJ{\cS_p(\cA;\cI,\cJ)}
\def\SABIJ{\cS_p(\cA,\cB;\cI,\cJ)}
\def \xbar{\overline x_p}

%
%\title[Bilinear sums of Kloosterman sums]{Bilinear sums of Kloosterman sums modulo a prime power}

\title[Level curves of rational functions]{Level curves of rational functions and unimodular points on rational curves}

\author[F. Pakovich]{Fedor Pakovich}
\address{Department of Mathematics, 
Ben Gurion University of the Negev, P.O.B. 653, Beer Sheva,  8410501, Israel}
\email{pakovich@math.bgu.ac.il}

 \author[I. E. Shparlinski] {Igor E. Shparlinski}

\address{Department of Pure Mathematics, University of New South Wales,
Sydney, NSW 2052, Australia}
\email{igor.shparlinski@unsw.edu.au}

\begin{abstract}
We obtain an improvement and broad generalisation of a result 
of N.~Ailon and Z.~Rudnick (2004) on common zeros of shifted powers of 
polynomials. Our approach is based on reducing this question to a  more general 
question of counting intersections of level curves of complex functions. We treat this question 
via classical tools of complex analysis  and algebraic geometry.
 \end{abstract}


\maketitle

\section{Introduction}

Recall that  Ailon and Rudnick~\cite[Theorem~1]{AR} have shown that for any multiplicatively independent polynomials  $P_1(z)$ and $P_2(z)$
with complex coefficients there exists a polynomial  $F(z) \in \C[z]$ such that 
for any positive integer $k$  the  greatest common divisor of $ P_1(z)^k-1$ and $P_2(z)^k-1$ divides $F$, that is   
$$
\gcd\(P_1(z)^k-1, P_2(z)^k-1\) \mid F(z), \quad k=1,2, \ldots.
$$

 Since it is easy to see that for a non-trivial polynomial
$P(z) \in \C[z]$ the  multiplicity of 
any factor of $P(z)^k-1$ does not exceed $\deg P$, 
the theorem of  Ailon and Rudnick is equivalent to
the following statement: 
 if $P_1$ and $P_2$ are complex polynomials, then   
\begin{equation}
 \label{eq:Zero AR} 
\#\bigcup_{k=1}^\infty \{z \in \C~:~P_1(z)^k = P_2(z)^k =1\} \le C(P_1,P_2),
\end{equation}
for some   constant $C(P_1,P_2)$ that depends only on $P_1$ and $P_2$, unless for some non-zero integers $m_1$ and $m_2$  we have
%with $(m_1,m_2) \ne (0,0$ we have
\begin{equation}
 \label{eq:P} 
P_1^{m_1}(z)P_2^{m_2}(z)=1
\end{equation}
identically.
Different versions and generalization Ailon-Rudnick result have been studied in many recent papers (see, for example,~\cite{GhHsTu2,HsTu,Ost,PaWa} and the references therein).

The method of Ailon and Rudnick~\cite{AR} relies 
on a result conjectured by
 Lang and proved by Ihara, Serre and Tate,
which states that the intersection of an irreducible curve $\cC$ in $\C^*\times \C^*$
with the roots of unity $\mu_{\infty}\times \mu_{\infty}$ is finite, unless $\cC$ is of the form
$X^nY^m - \eta= 0$ or $X^m - \eta Y^n = 0$, where $\eta\in  \mu_{\infty}$, that is unless $\cC$
is a translate by a torsion point of an algebraic subgroup  of  $\C^*\times \C^*$
 (see \cite{Lang}, \cite{Lang2}, and also \cite{BeSm}).

Corvaja, Masser, and Zannier in the paper of ~\cite{CMZ} 
ask about a possible extension of the Lang 
statement~\cite{Lang}, where instead of the intersection of $\cC$ with $\mu_{\infty}\times \mu_{\infty}$ the intersection with $S^1\times S^1$  is considered (here 
$S^1$ is treated as the topological closure of 
torsion points). In particular, 
they proved that the system
\begin{equation}
 \label{eq:Z} 
\left| z \right | =\left| P(z)\right | =1,
\end{equation} where $P(z)$ is a polynomial, has finitely many  solutions, unless $P(z)$ is a monomial. 
They also remarked that if $P(z)$ is allowed to be a rational function, then the system~\eqref{eq:Z} might have infinitely many solutions for non-monomial $P(z)$. 

 In this paper we consider  
the system of equations 
for the level curves 
\begin{equation}
 \label{eq:Unimod} 
\left| P_1(z)\right | =\left| P_2(z)\right | =1,
\end{equation}
where  $P_1(z)$ and $P_2(z)$ arbitrary  rational functions, 
generalising the systems~\eqref{eq:Zero AR} and ~\eqref{eq:Z}. 
Using classical tools of complex analysis and algebraic geometry,
 we describe $P_1$ and $P_2$ for which this system has infinitely many solutions and 
provide bounds for the number of  solutions in the other cases. 
Thus, our results can be considered as extensions of the result of Ailon and Rudnick~\cite{AR}
as well as of the Lang statement~\cite{Lang} in the particular case concerning of curves of genus zero.

\section{Results}

Recall that a {\it finite Blaschke product} is a rational function $B(z) \in \C(z)$ of the form 
$$B(z)=\zeta\prod_{i=1}^n\left(\frac{z-a_i}{1-\bar{a_i}z}\right)^{m_i},
$$ where $a_i$ are complex numbers in the open unit disc 
$$\D   = \{z \in \C~:~|z| <1\},$$
the exponents $m_i$, $i=1, \ldots, n$, are positive integers, and $\left| \zeta\right |=1$. 
A rational function 
$Q(z)$ of the form $Q(z) = B_1(z)/B_2(z)$, where 
$B_1$ and $B_2$ are finite Blaschke products,  is called a {\it quotient of finite Blaschke products}.

In the above notation,  our first result is the following.

\begin{theorem}
 \label{thm:LC-GenZero}
 Let $\cC: F(x,y)=0$, where $F(x,y) \in \C[x,y]$, 
be an irreducible algebraic curve of genus zero and of degree $d= \deg F$. Then $\cC$ has at most $d^2$ unimodular points, unless it can be parametrised 
by some quotients of finite Blaschke products 
 $x = Q_1(z)$ and $y = Q_2(z)$.
\end{theorem}

Our second result is the following generalisation of the bound~\eqref{eq:Zero AR}.

\begin{theorem}
 \label{thm:cor}
Let $P_1(z)$ and $P_2(z)$ be complex rational functions of degrees $n_1$ and $n_2$. Then 
\begin{equation}
 \label{eq:Zero PS} 
\begin{split}
\# \{z \in \C~:~\left| P_1(z)\right | =\left| P_2(z)\right | =1\} \le (n_1+n_2)^2,
\end{split}
\end{equation}
unless
\begin{equation}
 \label{eq:PBW} 
 P_1=B_1\circ W \mand P_2=B_2\circ W
\end{equation}
 for some  quotients of finite Blaschke products $B_1$ and $B_2$ and 
rational function $W$. 
\end{theorem}

In order to see that Theorem~\ref{thm:cor} implies~\eqref{eq:Zero AR} it is enough to observe that if
a quotient of finite Blaschke products is a polynomial, then this polynomial is necessary a power. Thus, ~\eqref{eq:PBW} reduces to 
$$P_1=W^{m_1} \mand   P_2=W^{m_2},$$
implying~\eqref{eq:P}.

Notice that since any 
quotient of finite Blaschke products  maps 
the unit circle
$$
\T = \{z \in \C~:~|z| =1\}
$$
to itself, if $P_1$ and $P_2$ satisfy~\eqref{eq:PBW}, then the level curves  $\left| P_1(z)\right | =1$ and $\left| P_2(z)\right | =1$ have a common component 
$W^{-1}\{\T\}$, so  a bound like~\eqref{eq:Zero PS}, 
or any other finiteness result, 
cannot exist. In particular, this happens if $P_1$ is a unimodular constant and 
$P_2$ is an arbitrary rational function (in this case~\eqref{eq:PBW} 
holds for $B_1=P_1$, $B_2=z$, and $W=P_2$).

\section{Proofs}
\label{sec:curve}

Since the inverse Cayley transform
$$z \mapsto T(z) = i\frac{1+z}{1-z}$$ 
 maps
$\D$ to the upper half-plane, and  the unit circle $\T$ maps under $T$ to the extended real line, a rational function   
$Q$ is a quotient of a finite Blaschke product if and only if the rational function 
$$R=T\circ Q\circ T^{-1}$$ maps $\R\cup \infty$ to $\R\cup \infty$. 
In turn, the last condition is equivalent 
 to the condition that $R$ has real coefficients (since 
 $R(z)$ and $\overline{R}(z)$ 
 coincide for infinitely many values of $z$).

 Thus, Theorem~\ref{thm:LC-GenZero} is equivalent to the following statement.

\begin{theorem}
 \label{lem:RealPoint Curve}
If an irreducible algebraic curve 
 $\cC: F(x,y)=0$ of genus zero and degree $d$ has more than $d^2$ real points, 
then $\cC$ can be para\-met\-rised by rational functions with real coefficients.  
\end{theorem}

\begin{proof}Observe that real points of $\cC$ belong to the intersection of the 
curve  $\cC$ and the curve $\overline{\cC}:\overline{F}(x,y)=0$. Therefore, it follows from the B\'ezout theorem that whenever $\cC$ has more than $d^2$ real points 
there exists 
$c\in \C$ such that $\overline{F}=c F$.  Such $c$ must satisfy $c\overline{c}=1$, implying that we can find a complex number 
$\lambda$ such that $\lambda^2=c$ and $\lambda\overline{\lambda} =1$. Since
$$\overline{\lambda F}=\overline{\lambda}\lambda^2 F=\lambda F,$$ the polynomial $\lambda F$ has real coefficients, and hence $\cC$ can be defined over $\R$. 

  Since the maximal number of singular points  of a plane curve of degree $d$ does not exceed $$\frac{(d-1)(d-2)}{2}$$ 
(see, for example,~\cite[Page~49]{fi}) and $\cC$ has more  than $d^2$ real points,  $\cC$ has a non-singular real point. 
Finally, an algebraic curve $\cC$ of genus zero defined over $\R$ admits a parametrisation by rational functions defined over $\R$ whenever 
$\cC$ has at least one non-singular $\R$-point (see, 
for example,~\cite[Theorem~7.6]{swd}). 
\end{proof}

In order to prove Theorem~\ref{thm:cor},
recall that  if a parametrisation $x=P_1(z)$,  $y=P_2(z)$ of an algebraic curve $C$ of genus zero is proper, that is, if 
$$
 \C(z)=\C(P_1(z),P_2(z)),
$$
then 
$$\deg P_1=\deg_yF \mand \deg P_2=\deg_x F,$$
(see, for example,~\cite[Theorem~4.21]{swd}).

Let now $P_1$ and $P_2$ be rational functions of degrees $n_1$ and $n_2$. Then the L\"uroth theorem implies that 
there exist a rational function $W$ and rational functions $Q_1$ and $Q_2$ such that the equalities~\eqref{eq:PBW} hold, and 
$$x=Q_1(z), \qquad y=Q_2(z),$$ 
is a proper parametrisation of an algebraic curve $C$ of degree 
at most $n_1+n_2$. 
Therefore, if~\eqref{eq:Zero PS} does not hold, then $Q_1$ and $Q_2$ are
quotients of finite Blaschke products by
Theorem~\ref{thm:LC-GenZero}.

%\section{Concluding comments}

\begin{remark}{\rm
We observe that  the  above argument 
%%  of Section~\ref{sec:curve}
provides a simple geometric criterion for a curve $\cC:G(x,y)=0$ to have infinitely many unimodular points. Namely, considering instead of the curve $\cC$ a curve $\widehat{\cC}: \widehat G(x,y)=0$, where 
$$\widehat G(x,y)=
G\(T(x),T(y)\),$$
 we reduce the question to the question about real points of $\widehat\cC$. On the other hand, it is easy to see that an algebraic curve  has infinitely many real points if and only if it is defined over $\R$ and has at least one simple $\R$-point.  
Indeed, the necessity has been  proved above. In the other direction, if a curve   defined over $\R$  has a simple $\R$-point, then the implicit function theorem implies that it has infinitely many $\R$-points.   }
\end{remark}

\section*{Acknowledgement}

The authors would like to thank Yuri Bilu and Umberto Zannier for 
their comments.

This work of I.~E.~Shparlinski. was supported  in part by  
the  Australian Research Council  Grants DP170100786 and DP180100201.

\end{document}